\documentclass[12pt]{article}
\usepackage[left=3cm,top=3.0cm,right=3cm,bottom=2.6cm]{geometry}

\usepackage[ansinew]{inputenc}
\usepackage{amscd,amsfonts,amsmath,amssymb,amsthm,enumerate,epsfig,float,graphics,graphicx,pst-all,yhmath}
\newcommand{\inte}{\operatorname*{int}}
\newcommand{\aff}{\operatorname*{aff}}
\newcommand{\bd}{\operatorname*{bd}}

\newcommand{\conv}{\operatorname*{conv}}

\newcommand{\R}{\mathbb{R}}

\newcommand{\Rt}{\mathbb{R}^3}
\newcommand{\Rn}{\mathbb{R}^{n}}
\newcommand{\Sn}{\mathbb{S}^{n-1}}

\newcommand{\Sd}{\mathbb{S}^{2}}

\newtheorem{lemma}{Lemma}

\newtheorem{theorem}{Theorem}

\newtheorem{remark}{Remark}

\newtheorem{conjecture}{Conjecture}

\title {On characteristic properties of the ellipsoid in terms of circumscribed cones of a convex body}
\author{E. Morales-Amaya$^{1}$\footnote{Corresponding author: E. Morales-Amaya}, G. Mondrag\'on$^{2}$  and \\J. Jer\'onimo-Castro$^{3}$  \\
\small{$^{1,2}$Universidad Aut\'onoma de Guerrero, M\'exico}\\
\small{$^{3}$Universidad Aut\'onoma de Quer\'etaro, M\'exico}\\
 \small{\texttt{$^{1}$emoralesamaya@gmail.com,$^{2}$hermanweyl@gmail.com}}\\
 \small{\texttt{$^{3}$jesusjero@hotmail.com}}
}

\begin{document}

\maketitle
\begin{abstract}  
In order to prove two important geometrical pro\-blems in convexity, namely, the Conjecture of Bianchi and Gruber \cite{bigru} and the Conjecture of Barker and Larman \cite{Barker}, it is necessary to obtain new characteristic properties of the ellipsoid, which involves the notions defined in such problems. In this work we present a series of results which intend to be a progress in such direction: Let $L,K\subset \Rn$ be convex bodies, $n\geq 3$, and $L\subset \inte K$. Then each of the following conditions i), ii) and iii) implies that $L$ is an ellipsoid. 

i) $L$ is $O$-symmetric and, for every $x\in \bd K$, the support cone $S(L,x)$ is ellipsoidal.

ii) there exists a point $p\in \Rn$ such that for every $x\in \bd K$, there exists  $y\in \bd K$ and hyperplane $\Pi$, passing through $p$,  such that
\[
 S(L,x)\cap S(L,y)=\Pi \cap \bd K. 
\]

iii) $K$ and $L$ are $O$-symmetric, every $x\in \bd K$ is a pole of $L$ and $\Omega_x:=S(L,x)\cap S(L,-x)$ is contained in $\inte K$.   

In the case ii), $K$ is also an ellipsoid and it is concentric with $L$. On the other hand, let $K\subset \Rn$ be an $O$-symmetric convex body, $n\geq 3$, and let $B\subset \inte \Rn$ be a ball with centre at $O$. We are going to prove that if $B$ is small enough and all the sections of $K$ given by planes   tangent to $B$ are $(n-1)$-ellipsoids, then $K$ is an $n$-ellipsoid.
 \end{abstract}
\section{Introduction}
The important, and beautiful, characterization of the ellipsoid of Blaschke-Marchaud \cite{kreis}, \cite{marchaud}: 

\textit{If all the shadow boundaries of a convex body $K\subset \Rn$, $n\geq3$, are planar, then $K$ is an ellipsoid,} 

and the False Centre Theorem (FCT) of Aitchison, Petty, Rogers and Larman \cite{false_centre}, \cite{larman}:

 \textit{If all the 2-sections of a convex body $K\subset \Rn$, $n\geq3$, passing through a point $p$, are centrally symmetric, then either $K$ has centre at $p$ or is an ellipsoid,} 
 
have turned out to be very useful tools to prove that a convex body with a certain geometric condition, in particular, in terms of section or projections, is an ellipsoid. For example, on the one hand, Blaschke's characterization of the ellipsoid was used to prove that \textit{a normed space is Euclidean if orthogonality is symmetric} (See Pag. 103 of \cite{bus}), on the other hand, we can mention that a variant of Blaschke's characterization is used in the short proof of the FCT \cite{Falso_MM}. On the other hand, the FCT was applied in the sphere characterization of Montejano in terms of similar sections passing through a fix point \cite{monti}. To the reader interested in more results in relation to the ellipsoid, we recommend the classic article \cite{Petty}.

Thus, from our point of view, it is necessary to find new characterizations of the ellipsoid whose conditions allow us to address interesting problems that have remained unsolved for years. For example, between such problemas, we can mention the Conjecture of Bianchi and Gruber \cite{bigru}. In order to present such conjecture we will give the following definition.

Let $C\subset\mathbb R^n$ be a convex cone with apex $x$. We say that $C$ is a symmetric cone with axis $L_x$, if there exists a line $L_x$ through $x$ such that for every $2$-dimensional plane $\Gamma$ which contains $L_x$, it holds that $L_x$ is the angle bisector of the angular region $\Gamma\cap C$.

\textbf{Conjecture of Bianchi and Gruber}: \textit{Let $K$ be a convex body contained in the interior of the unit ball $B(n)$, $n \geq 3$. If for every $x\in \mathbb{S}^{n-1}$, $C(K,x)$ is a symmetric cone, then $K$ is an ellipsoid,} 

where, for $x \in \Rn \backslash K$, $C(K, x)$ denotes the cone generated by $K$ with apex $x$, and, on the other hand, the Conjecture of Barker and Larman: 

\textit{Let $K \subset \Rn$ be a convex body, $n \geq 3$. If there exists a ball $B\subset \inte K$ such that, for every supporting hyperplane $\Pi$ of $B$, the section $\Pi \cap K$ is centrally symmetric, then $K$ is an ellipsoid.}
 
Let $\mathbb{R}^{n}$ be the Euclidean space of dimension $n$ endowed with the usual inner product $\langle \cdot, \cdot\rangle : \mathbb{R}^{n} \times \mathbb{R}^{n} \rightarrow \R$. We take an orthogonal system of coordinates $(x_1,...,x_{n})$ for  $\mathbb{R}^{n}$. Let $B_r(n)=\{x\in \mathbb{R}^{n}: ||x||\leq r\}$ be the $n$-ball of radius $r$ centered  at the origin, and let $r\mathbb{S}^{n-1}=\{x\in \mathbb{R}^{n}: ||x|| = r\}$ be its boundary. 

Let $x,y,z\in \Rt$. We denote by $l(x,y)$ the line defined by $x$ and $y$. On the other hand, if the points $x,y,z$ are not collinear, we denote by $P(x,y,z)$ the plane defined by $x,y,z$.

Let $K\subset \Rn$ be a convex body. Given a point $x \in \Rn \backslash K$ we denote the cone generated by $K$ with apex $x$ by $C(K, x)$, that is, $C(K, x) := \{x + \lambda(y - x) : y \in K, \lambda \geq  0\}$, by $S(K, x)$ the boundary of $C(K, x)$, in other words, $S(K, x)$ is the support cone of $K$ from the point $x$ and by 
$\Sigma(K, x)$ the \textit{graze} of $K$ from $x$, that is, $\Sigma(K, x) :=S(K, x)\cap  \bd K$. A very especial case is when the apexes of the cones are points at infinity. In this case the grazes are called \textit{shadow boundaries} and they are obtained by intersections of $\bd K$ with circumscribed cylinders. The shadow boundary with respect to the direction $u$ (the line $L$) will be denoted by $S\partial (K,u)$ ($S\partial (K,L)$).

Let $C$ be a convex cone, the cone is said to be \textit{ellipsoidal} if it has one hyper-section which is an ellipsoid.

Our first result in this work is the following theorem.
\begin{theorem}\label{camioneta}
Let $L,K\subset \Rn$ be convex bodies, $n\geq 3$. Suppose that 
$L\subset \inte K$, $L$ is $O$-symmetric and, for every $x\in \bd K$, the cone $C(L,x)$ is ellipsoidal. Then $L$ is an ellipsoid.
\end{theorem}

If $K\subset \Rn$ is a  $O$-symmetric body, $O$ is the origin of a coordinate system and $x\in \Rn\backslash K$, then the intersection of the support cones of $K$ with centers at $x$ and $-x$ is a $O$-symmetric set. \textit{What can we say of $K$ if, in addition, we assume that such intersection is contained in a hyperplane?}

Such question was motivated by the proof of the main result of \cite{gjmc2}. In Theorem \ref{jetta} we present a characterization of the ellipsoid which can be considered as a particular case to the previous question. We use Theorem \ref{camioneta} to prove Theorem \ref{jetta}.

 Let $L,K\subset \Rn$ be $O$-symmetric convex bodies, $n\geq 3$. Suppose that every $x\in \bd K$ the graze $\Sigma(L,x)$ is contained in a hyperplane. In the proof of the main theorem of \cite{gjmc2}, it was observed that the set $\Omega_x:=S(L,x)\cap S(L,-x)$ is contained in an hyperplane (See Lemma \ref{boquita}).  This observation motives the following problem:
\begin{conjecture}\label{tere}
Let $L,K\subset \Rn$ be convex bodies, $n\geq 3$. Suppose that $L\subset \inte K$ and, for every $x\in \bd K$, there exists  $y\in \bd K$ and hyperplane $\Pi$  such that
\begin{eqnarray}\label{nevada}
S(L,x)\cap S(L,y)=\Pi \cap \bd K. 
\end{eqnarray} 
Then $L$ and $K$ are homothetic ellipsoids.
\end{conjecture}
In this work we present a particular case of the Conjecture \ref{tere}. Before we state such result, we give the next example which will show the main conditions involved. 
We denote by $N$ and $S$ the north and south poles of $\mathbb{S}^{n-1}$, that is, the points $(0,0, ..., 1)$ and $(0,0, ..., -1)$, respectively, and let $S_N$ denote the equator of $\mathbb{S}^n$. Let $B$ be the unique sphere inscribed in the convex body $\conv (\{N,S\}\cup S_N)$. By virtue of the symmetry of the sphere, it follows that, for every $x\in \mathbb{S}^n$, the relation
\begin{eqnarray}\label{lluvia}
S(B,x)\cap S(B,-x)= S_x. 
\end{eqnarray}
holds, where $S_x$ is the great circle of $\mathbb{S}^{n-1}$ perpendicular to $x$.

Let $A:\Rn \rightarrow \Rn$ be  an affine map. We denote by $E$ and $F$ the homothetic ellipsoid $A(\mathbb{S}^{n-1})$ and $A(B)$, respectively. Then, by (\ref{lluvia}), it follows that, for every $y\in E$, the relation
\[
S(F,y)\cap S(F,-y)= A(S_x)\subset E. 
\]
holds, where $x$ is such that $y=A(x)$.
The next results shows that the inverse statement holds.   
\begin{theorem}\label{jetta}
Let $L,K\subset \Rn$ be convex bodies, $n\geq 3$, and let $p\in \inte K$. Suppose that $L\subset \inte K$ and, for every $x\in \bd K$, there exists  $y\in \bd K$ and hyperplane $\Pi$, passing through $p$,  such that
\begin{eqnarray}\label{nevada}
S(L,x)\cap S(L,y)=\Pi \cap \bd K. 
\end{eqnarray} 
Then $L$ and $K$ are concentric homothetic ellipsoids.
\end{theorem}
 We complete $\mathbb{R}^n$ to $n$-dimensional projective space $\mathbb{P}^n$ by adding the hyperplane at infinity. We say that $O$ is a pole of $K$ if there is a hyperplane $H$ of $\mathbb{P}^n$ with the property that, for every line $L$ through $O$ such that $\bd K \cap L = \{A, B\}$, the cross ratio of $A, B, O$ and the intersection of $L$ and $H$ is minus one. That is,
\[
[A, B; O, L \cap H] = -1.
\]
If this is so, we say that $H$ is a polar hyperplane of $K$ and also that $H$ is the polar of the pole $O$ (See, for example, \cite{MMpolo}).

If $O \in  \inte K$ is a pole of $K$, then we say that $O$ is a \textit{projective centre of symmetry} of $K$ because, in this case, the polar of $K$ is a hyperplane $H$ that does not intersect $K$, and if $\pi$ is a projective isomorphism that sends $H$ to the
hyperplane at infinity, then $\pi(K)$ is a compact centrally symmetric convex body
with centre $\pi(O)$. 

If $O\in \mathbb{P}^n \backslash K$ is a pole of $K$, then its polar $H$ will be called a \textit{projective hyperplane of symmetry} of K. In this case, $H\cap \bd K\not= \emptyset$. It is easy to see that
\begin{eqnarray}\label{meche}
H\cap \bd K =\Sigma(K,O) 
\end{eqnarray}
If $\pi$ is a projective isomorphism that sends $\pi(K)$ into a compact convex body and $\pi(O)$ into a point at infinity, then $\pi(H)$ is an affine hyperplane of symmetry of $K$ (See the definition of affine hyperplane of symmetry at the beginning of the second paragraph of section 2).

Let $E\subset \mathbb{R}^n \subset \mathbb{P}^n$ be an ellipsoid. Then every point $O\in \mathbb{P}^n \backslash \bd K$ is a pole of $E$ and every non-tangent hyperplane of $E$ is a polar of $E$.
\begin{theorem}\label{gina}
Let $L,K\subset \Rn$ be $O$-symmetric convex bodies, $n\geq 3$. Suppose that every $x\in \bd K$ is a pole of $L$ and $\Omega_x:=S(L,x)\cap S(L,-x)$ is contained in $\inte K$. Then $L$ is an ellipsoid.
\end{theorem}
\begin{theorem}\label{amanecer}
Let $K\subset \Rn$ be $O$-symmetric convex body, $n\geq 3$, and let $B\subset \inte \Rn$ be a ball with centre at $O$. Suppose that, for every hyperplane $\Pi$ tangent to $B$: (1) the section $\Pi\cap K$ is an ellipsoid,  
(2) the ball $B$ is contained in the interior of the cylinder $H_{\Pi}:=\conv [(\Pi\cap K)\cup (-\Pi\cap K)]$. Then $K$ is an ellipsoid.
\end{theorem}

In this work we are going to present 6 characterizations of the ellipsoid. The theorems \ref{camioneta} and \ref{amanecer}, in our opinion, are generalizations of the Lemma 2 of \cite{chakerian} and theorem 16.12 from Busemann's book \cite{bus}, respectively. The Theorem \ref{camioneta} and the Theorem \ref{amanecer} are proved in sections 2 and 5, respectively. In Theorem \ref{camioneta}, in order to prove that a convex body $L\subset \Rn$ is an ellipsoid, instead of taking into account \textit{ellipsoidal circumscribed cylinders} of $K$, we consider \textit{ellipsoidal circumscribed cones} of $L$ with vertices in another convex body $K$ such that $L\subset \inte K$.  On the other hand, in Theorem \ref{amanecer}, instead of taking into account concurrent sections of a convex body $L\subset \Rn$ that are ellipsoids, we consider tangent sections to a ball $B$, $L\subset \inte B$, that are ellipsoids. In both Theorem \ref{camioneta} and Theorem \ref{amanecer} we have assumed that the convex body $L$ is centrally symmetric because we are interested in finding a solution to the most symmetric cases of the Conjecture of
Bianchi-Gruber and of Conjecture the Barker-Larman. 
 Although Theorem \ref{camioneta} is a special case of Theorem 2 of 
 \cite{bigru}, in this work we present a direct proof, that is, without resorting to polar duality. We believe that this point of view can be useful for a solution to the aforementioned conjectures, as has been shown by the proof of the main theorem of \cite{gjmc2}.

The Theorem \ref{jetta} and the Theorem \ref{gina} are proved in sections 3 and 4, respectively.
\section{Characterizations of the ellipsoid in terms of e\-llip\-soidal circumscribed cones of a convex body}
Before giving the proof of the Theorem \ref{camioneta} we are going to present some auxiliary results. The first of them is in relation to three properties of ellipsoidal cones. In the next two we are going to assume that the hypothesis of the Theorem \ref{camioneta} is fulfilled. The second result refers to the intersection of cones with vertices at points symmetrical with respect to the center of the convex and the third result will help us demonstrate that the grazes of the convex are flat.   
  
Let $K\subset \Rn$ be either a convex body or a convex cone, let $H\subset \Rn$ be an hyperplane and let $u\in \Sn$. The hyperplane is said to be an \textit{affine hyperplane of symmetry of $K$} if every chord of $K$, parallel to $u$, has its mid-point in $H$.  
\begin{lemma}\label{gatita}
Let $C(x)\subset \Rn$ be convex cone with apex at $x\in \Rn$. If $C(x)$ is ellipsoidal the following properties holds:
\begin{itemize}
\item [i)] All the bounded hyper-sections of $C(x)$ are ellipsoids.
\item [ii)] Every hyperplane passing through $x$ and with interior points of $C(x)$ is an affine hyperplane of symmetry of $C(x)$.
\item [iii)] If $L$ is a ray passing through $x$ and it is contained in $\inte C(x)$, then there exists an hyperplane $\Pi$ such that the ellipsoid $\Pi \cap C(x)$ has centre at $\Pi \cap L$.
\end{itemize}
\end{lemma} 
\begin{proof}
We will present only the proof of iii) since the proofs of i) and ii) follows immediately from the well known properties of poles and polar. For the proof of iii), let us suppose that $\Rn$ is an affine chart of the projective space $\mathbb{P}^n$. Let $p\in L$, $p\not=x$, and let $H$ be a hyperplane such that $p\in H$ and $H\cap C(x)$ is bounded. By i), the section $H\cap C(x)$ is an ellipsoid. Let $L_p$ be the polar of $p$ with respect to the ellipsoid $H\cap C(x)$. We denote by $\Gamma$ the hyperplane defined by $x$ and $L_p$. Let $\Pi$ be the hyper-plane parallel to $\Gamma$ and passing through $p$. We claim that $\Pi$ is the hyper-plane that we are looking for. In fact, the $(n-2)$-plane $L_{\infty}:=\Pi\cap \Gamma$ is the $(n-2)$-plane at the infinite in $\Pi$ whose pole with respect to the ellipsoid $\Pi \cap C(x)$ is $p$. Thus $\Pi \cap C(x)$ has center at $p$.
\end{proof} 
\begin{remark}\label{frank}
Notice that, with the notation of Lemma \ref{gatita}, the directions corresponding to the affine symmetries with respect to the hyper-planes containing $L$ are parallel to $\Pi$.
\end{remark}
Let $x\in \bd K$. We take a system of coordinates such that $O$ is the origin. By the hypothesis the cone $C(x)$ is an ellipsoidal cone. Since $L$ is $O$-symmetric the cone $C(-x)$ is an ellipsoidal cone too.

\textbf{Claim.} The curve $\Delta_x:=C(L,x)\cap C(L,-x)$ is a planar curve, i.e., there exists a plane $\Lambda_x$ such that $\Delta_x\subset \Lambda_x$.
\begin{proof}
By iii) of Lemma \ref{gatita}, there exist a plane $\Lambda_x$, through the $O$, such that $\Lambda_x \cap C_x$ has center at $O$. Thus $C(L,x)\cap C(L,-x)\subset \Lambda_x$. 
\end{proof}
\begin{lemma}\label{vane}
Let $x_1,x_2\in \bd K$ such that $l(x_1,x_2)\cap L=\emptyset$. Let $T_a,T_b$ be the supporting planes of $L$ containing $l(x_1,x_2)$ and making contact with $\bd L$ at $a,b$, respectively. Then $l(a,b)$ is parallel to $\Lambda_{x_1}\cap \Lambda_{x_2}$.
\end{lemma}
\begin{proof}
By Remark \ref{frank}, the directions corresponding to the affine reflection relatives to the planes containing $l(O,x_i)$, which leave $C_{x_{i}}$ invariant, are parallel to $\Lambda_{x_i}$, $i=1,2$. Thus the plane $G$ defined by $O$ and $l(x_1,x_2)$ is plane of affine symmetry of $C_{x_1}$ and $C_{x_2}$ with respect to the direction 
$\Lambda_{x_1}\cap \Lambda_{x_2}$ (See Fig. \ref{preciosisima}). It follows that the planes $P(a,b,x_1)$ and $P(a,b,x_2)$ are parallel to $\Lambda_{x_1}\cap \Lambda_{x_2}$. Hence $l(a,b)$ is parallel to $\Lambda_{x_1}\cap \Lambda_{x_2}$.
\end{proof}
\begin{figure}[H]
    \centering
    \includegraphics[width=.92\textwidth]{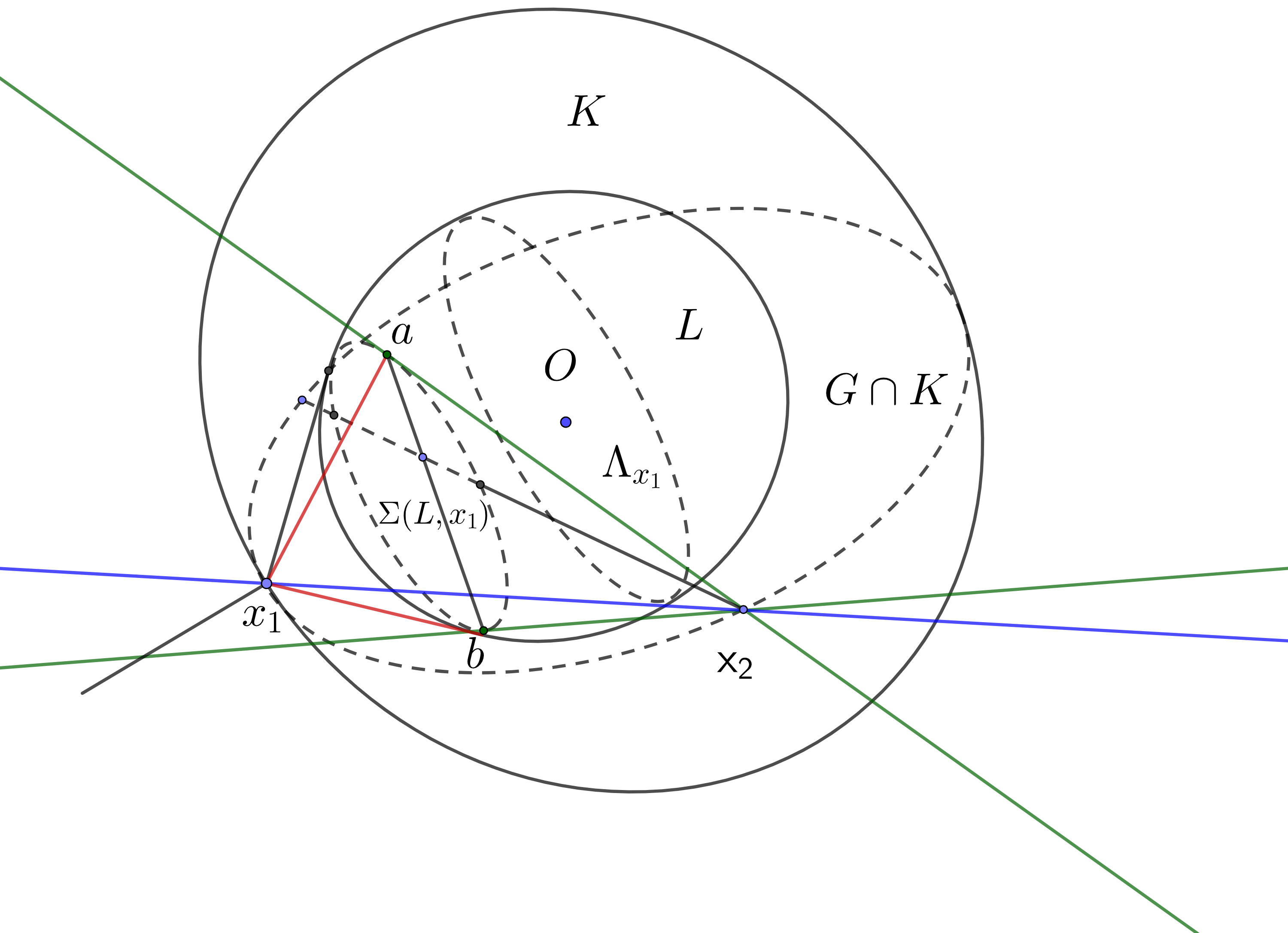}
    \caption{The chord $l(a,b)$ is parallel to $\Lambda_{x_1}\cap \Lambda_{x_2}$.}
    \label{preciosisima}
\end{figure}
\textbf{Proof of Theorem \ref{camioneta}.} 
It will be enough to consider the case $n=3$ since if $n>3$, on the one hand, every hyper-sections of $K$ and $L$ passing through $O$ satisfies the conditions of the Theorem \ref{camioneta} and, on the other hand, if every $(n-1)$-section of $L$ through $O$ is an ellipsoid, $L$ is an ellipsoid.  

Case $n=3$. We are going to prove that, for $x\in \bd K$, the graze $\Sigma (L,x)$ is contained in a plane. By the theorem 5 of \cite{gjmc1} we will conclude that $L$ is an ellipsoid.  

Let $x\in \bd K$. Pick a point $a\in \Sigma(L,x)$. Let $\Lambda$ be a plane passing through $a$ and parallel to $\Lambda_x$. Let $b\in \Sigma(L,x)$, $a\not=b$. By Lemma \ref{vane}, $l(a,b)$ is parallel to $\Lambda_x$. Thus $l(a,b)\subset  \Lambda$, i.e., $b\in  \Lambda$. Consequently, $\Sigma (L,x)\subset \Lambda$. 
\section{Characterizations of the ellipsoid in terms of the intersection of two circumscribed cones of a convex body.}
\textbf{Proof  of Theorem \ref{jetta}.}
Before we present a proof of Theorem \ref{jetta}, we will give the following definitions. A chord $[a,b]$ of a convex body $K\subset \Rn$ is said to be an \textit{affine diameter} of $K$ provided there are two parallel, distinct hyperplanes $\Pi_a$ and $\Pi_b$ both supporting $K$ such that $a\in \Pi_a$ and $b\in \Pi_b$ (\cite{soltan}).  

Let $K\subset \mathbb{R}^2$ be a convex figure.  A parallelogram $P$ is said to be \textit{circumscribed} about $K$ if $K\subset P$, the four sides of $P$ have non empty intersection with $K$. Two affine diameters of $K$ are called \textit{conjugate} provided there is a parallelogram $P$ circumscribed about $K$ such that each side of $P$ is parallel to some of the diameters (\cite{soltan}).  A \textit{Radon curve} (See \cite{radon}) is defined as a regular affine image of a convex curve with centre $O$ in the plane, such that its polar is a $90^o$ rotation about $O$. For us, it will be important to observe that a centrally symmetric convex body $K$ in the plane is bounded by a Radon curve if and only if any affine diameter of $K$ has a conjugate affine diameter (which was proved in \cite{radon}).  

If $(X,\|.\|)$ is a normed linear space and if $x,y\in X$ then we say that $x$ is \textit{normal} to $y$ and write $x\dashv y$ if $\|x+\alpha y \|\geq \|x\|$ for all $\alpha$. Geometrically this means that $x\dashv y$ if and only if the line $x+\alpha y$  ($\alpha \in \Rn$) supports the ball $B[o,\|x\|]$ at $x$.

The next result was proved in \cite{thompson}.
\begin{itemize}
\item[(T)] \textit {If $(X,\|.\|)$ is a normed linear space with $\dim X\geq 3$ then the normality is symmetric, i.e., if $x\dashv y$ then $y\dashv x$, if and only if the space is an inner product space.}
\end{itemize}

The next result will be used in the proof of Theorem \ref{jetta}, since we did not find a proof of it in the literature and since, from our point of view, such theorem is interesting by itself, we will give a short one which was suggested for one of the referees. 
\begin{theorem}\label{radon}
Let $K\subset \Rn$ be convex body $O$-symmetric, $n\geq 3$. Suppose that, for every 2-plane $\Pi$ passing through $O$, the section $\Pi \cap \bd K$ is a Radon curve. Then $L$ is an ellipsoid.
\end{theorem}
\begin{proof}
By the definition of Radon curve it follows that, in the normed linear space $X$ defined by $K$, the normality is symmetric. Thus, by (T), $X$ is an inner product space. Hence $K$ is an ellipsoid.
\end{proof}

\textbf{Proof of Theorem \ref{jetta} for $n=3$}. We will present the proof of Theorem \ref{jetta}, for the case $n=3$, by means of a series of arguments which will be enumerated.  

\begin{itemize}
\item [1)] $K$ is strictly convex. On the contrary, let us assume that $K$ is not strictly convex. Let $I\subset \bd K$ be an line segment. Let $\Gamma$ be a supporting plane of $L$ containing $I$. We pick $x\in (\Gamma \cap \bd K)\backslash I$. By virtue of the hypothesis, there exists a plane $\Pi$, though $p$, and a point $y\in \bd K$ such that the relation  
\begin{eqnarray}\label{calavera}
S(L,x)\cap S(L,y)=\Pi\cap \bd K
\end{eqnarray} 
holds. Since $\Gamma$ is a supporting plane of $S(L,x)$ relation (\ref{calavera}) implies that $\Gamma \cap \Pi$ is a supporting line of $\Pi\cap \bd K$ (We can interpret the hypothesis as if by projecting from $x$ and $y$ the convex body $L$ onto $\Pi$, respectively, we obtain the section $\Pi \cap K$, therefore, if $\Gamma$ is a support plane of $L$ that passes through $x$, then the line $\Gamma \cap \Pi$ is a support line of $\Pi \cap \bd K$). Thus $\Pi=\aff\{p,I\}$. We denote by $T$ the triangle defined by $x$ and $I$. Let $z\in \inte T$. Since the line $L(x,z)$ is supporting line of $L$ there exists a point $w \in \bd L \cap L(x,z)$ in the interior of $T$. Let $u\in (\Gamma \cap \bd K)\backslash I$ and $u\not=x$ and let $v:=L(u,w)\cap \bd (\Gamma \cap \bd K)$. Again, by virtue of the hypothesis, there exists a plane $\Pi'$, though $p$, and a point $u'\in \bd K$ such that the relation  
$S(L,u)\cap S(L,u')=\Pi'\cap \bd K$ holds. However, as we have seen for $x$, the plane $\Pi'$ is equal to $\Pi$. On the other hand, $v\in \Pi \cap \bd K$. Hence $v\in \Pi \cap \bd K \cap \Gamma$. On the other hand, since $\Pi \cap \bd K \cap \Gamma= I$ it follows that $v\in I$ but, by the choice of $u$, we have that $v\notin I$. This contradiction shows that $K$ is strictly convex. 

\item [2)] Let $x,y\in \bd K$ be points and let $\Pi$ be a plane such that they satisfies the relation (\ref{nevada}). Then there exists supporting planes $\Pi_x$ and $\Pi_y$ of $K$ at $x$ and $y$, respectively, parallel to $\Pi$. 
\end{itemize}
 
Let $[u,v]$ be a diametral chord of $Z:=\Pi \cap \bd K$. By virtue that $l(x,u)$ and $l(x,v)$ are supporting lines of 
$L$ there exists $a,b\in \bd L$ such that $a\in [x,u]$ and $b\in [x,v]$. Let $\Pi_u, \Pi_v$ be supporting planes of $L$ such that $l(x,u)\subset \Pi_u$ and $l(x,v)\subset \Pi_v.$ We claim that $\Pi_u \cap \Pi_v$ is supporting line of $K$. Otherwise, there would exist $z\in \bd K$, $z\not=x$ and $z\in \Pi_u \cap \Pi_v$. Then, since the line $l(z,u)\subset \Pi_u$ is supporting line of $L$, it would exist a point $w\in \bd L$ with $w\in l(z,u)$, $w\not=a$. However, this it would imply that in the supporting line $\Pi \cap \Pi_u$ of $Z$ it would exist a point $t:=l(x,w)\cap(\Pi\cap \Pi_u)$ of $Z$, $t\not=u$, nevertheless, it would contradict the strictly convexity of $K$. By virtue that $\Pi_u \cap \Pi_v$ is parallel to $\Pi$ and it is passing through $x$, varying the diametral chords of $Z$ the lines $\Pi_u \cap \Pi_v$ generated the supporting plane $\Pi_x$ of $K$ at $x$ parallel to $\Pi$. With an analogous argument we deduce the existence of a supporting plane $\Pi_y$ of $K$ at $y$ parallel to $\Pi$.

\begin{itemize}
\item [3)] Let $\Gamma$ be a plane containing the line $l(x,y)$ and let $[u,v]:=\Gamma \cap Z$. Then there exist a plane $\Delta$ through $p$ such that $S(L,u)\cap S(L,v)=\Delta \cap \bd K$ and $l(x,y)\subset \Delta$.
\end{itemize}
Let $\bar{u}\in \bd K$ be a point and let $\Delta$ be a plane such that $p\in \Delta$ and the relation $S(L,u)\cap S(L,\bar{u})=\Delta \cap \bd K$ holds. By virtue of the hypothesis of the Theorem \ref{jetta} such $\bar{u}$ and $\Delta$ exists. Since $l(x,u)$ and $l(y,u)$ are supporting lines of $L$ it follows that the plane $\Delta$ contains the line $l(x,y)$. Let $l$ be supporting lines of $\Pi\cap L$ passing through $u$ and let $w:=l\cap \bd K$. Then $\Delta=P(x,y,w)$ and $\bar{u}=v$ (the supporting line $m$ of $\Pi\cap L$ passing through $w$, $m\not=l$, must pass, on the one hand, through the point $\bar{u}$ and, on the other hand by the point $l(p,u) \cap \bd K$ but this point is $v$).
\begin{itemize}
\item [4)] $K$ is centrally symmetric with center at $p$.
\end{itemize}
By the arbitrariness of $\Gamma$ in 3), we conclude that $p\in l(x,y)$. Thus $p=l(x,y)\cap \Pi$. By 2), there are supporting planes $G_u$, $G_v$ of $K$ at $u,v$, respectively, parallel to $\Delta$. Hence there are supporting parallel lines $G_u\cap \Pi$, $G_v \cap \Pi$ of the curve $Z$ at $u,v$. Varying the plane $\Gamma$, keeping the condition $l(x,y)\subset \Pi$, we obtain the every chord of $Z$ through $p$ is an affine diameter. By Theorem 4.1 of \cite{soltan}, it follows that $Z$ is centrally symmetric with center at $p$. Now varying the plane $\Pi$, such that $l(u,v)\subset \Pi$, we conclude that $K$ is centrally symmetric. 
\begin{itemize}
\item [5)] For every plane $\Pi$ passing through $p$ the section $\Pi \cap \bd K$ is a Radon curve. 
\end{itemize}
We take a system of coordinates with $p$ as the origin. Let $\Pi$ be a plane through $p$. By an argument of continuity, there exist $x,y\in \bd K$ such that for $x,y$ and $\Pi$ the relation (\ref{nevada}) holds. Let $[u,-u]\subset \Pi$ be a chord of $Z:=\Pi \cap \bd K$. By 2), there exist a plane $\Delta$ through $p$ such that $S(L,u)\cap S(L,-u)=\Delta \cap \bd K$ and $l(x,y)\subset \Delta$. By 2), there are supporting planes $G_u$, $G_{-u}$ of $K$ at $u,-u$, respectively, parallel to $\Delta$. Hence there are supporting lines $G_u\cap \Pi$, $G_{-u} \cap \Pi$ of the curve $Z$ at $u,-u$ parallel to $\Delta$. Let $[w,-w]:=\Delta \cap Z$. By 2), the points $w,-w$ and the plane $\Gamma:=P(x,y,u)$ are such that the relation $S(L,w)\cap S(L,-w)=\Gamma \cap \bd K$
holds. Then by 2), there are supporting planes $H_w$, $H_{-w}$ of $K$ at $w,-w$, respectively, parallel to $\Gamma$. Hence there are supporting parallel lines $H_w\cap \Pi$, $H_{-w} \cap \Pi$ of the curve $Z$ at $w,-w$ parallel to $\Gamma$. That is, the chords $[u,-u]$ and $[w,-w]$ are conjugate diameters. Thus $Z$ is a Radon curve. 

By Theorem \ref{radon}, $K$ is an ellipsoid. 
\begin{itemize}
\item [6)] For every $x\in \bd K$, the cone $S(L,x)$ is an elliptical cone. 
\end{itemize}
For $x \in \bd K$, there exists a plane $\Pi$, $p\in \Pi$, such that  $S(L,x)\cap S(L,-x)=\Pi \cap \bd K$. Since $K$ is an ellipsoid, $\Pi \cap \bd K$ is an ellipse, and $S(L,x)$ is an elliptical cone. 
\begin{itemize}
\item [7)] $L$ is strictly convex.
\end{itemize}
Suppose that $L$ is not strictly convex. Let $[a,b]\subset \bd L$ be a line segment. Let $H$ be a supporting hyperplane of $L$ containing $[a,b]$. Let $x \in H\cap K$ which does not belong to the line $l(a,b)$. On the one hand, by 6), $C(L,x)$ is an ellipsoidal cone. On the other hand, since $H$ is a supporting plane of $C(L,x)$, 
$[a,b]\subset C(L,x)$ and the line $l(a,b)$ is not passing through the apex $x$ of   
$C(L,x)$ but this is imposible by virtue that the ellipses are strictly convex.
\begin{itemize}
\item [8)] $L$ is centrally symmetric with center at $p$.
\end{itemize}
8) follows immediately from 7) and the following result which was proved in \cite{papiefen}:
  
\textit{Let $L$ be a strictly convex body in the Euclidean space $\Rn$, $n \geq 3$, and let $K$ be a hypersurface which is the image of an embedding of the sphere $\mathbb{S}^{n-1}$, such that $L$ is contained in the interior of $K$. Suppose that, for every $x \in K$, there exists $y \in K$ such that the support cones of $L$ with apexes at $x$ and $y$ differ by a central symmetry. Then $L$ and $K$ are centrally symmetric and concentric.}

By 6) and 8) and Theorem \ref{camioneta}, $L$ is an ellipsoid. 
$\square$ 

\textbf{Proof of Theorem \ref{jetta} for $n>3$}.  

\begin{lemma}\label{marisol}
If $x,y\in \bd K$ and the hyperplane $\Pi$ are such that the relation 
\[
S(x,L)\cap S(y,L)=\Pi \cap K
\]
holds, then $x,y$ and $p$ are collinear. 
\end{lemma}
\begin{proof}
Let us suppose the opposite to the statement of Lemma \ref{marisol}, that is, let us suppose that $y\not=y':=L(p,x)\cap \bd K$, $y'\not=x$. We denote by $H$ the 2-plane defined by $x,y$ and $y'$ and let $u,v$ be the points given by the intersection of the line $\Pi \cap H$ with $\bd K$ (see Fig. \ref{end}).
Let $u'\in \bd K$ and let $\Pi'$ be a hyperplane such that  
\[
S(u,L)\cap S(u',L)=\Pi' \cap K.
\]
Since $u,v\in \Pi \cap K$ the lines $L(x,u), L(x,v), L(y,u)$ and $L(y,v)$ are supporting lines of $H\cap L$. This implies that $u'=v$ and $x,y\in \Pi'$. Thus $p\in \Pi'\cap H=L(x,y)$. On the other hand, $p\in L(x,y')$. Hence $p=x$ which is absurd. This contradiction shows that $y=y'$.
\end{proof}
\begin{figure}[H]
\includegraphics[width=.89\textwidth]{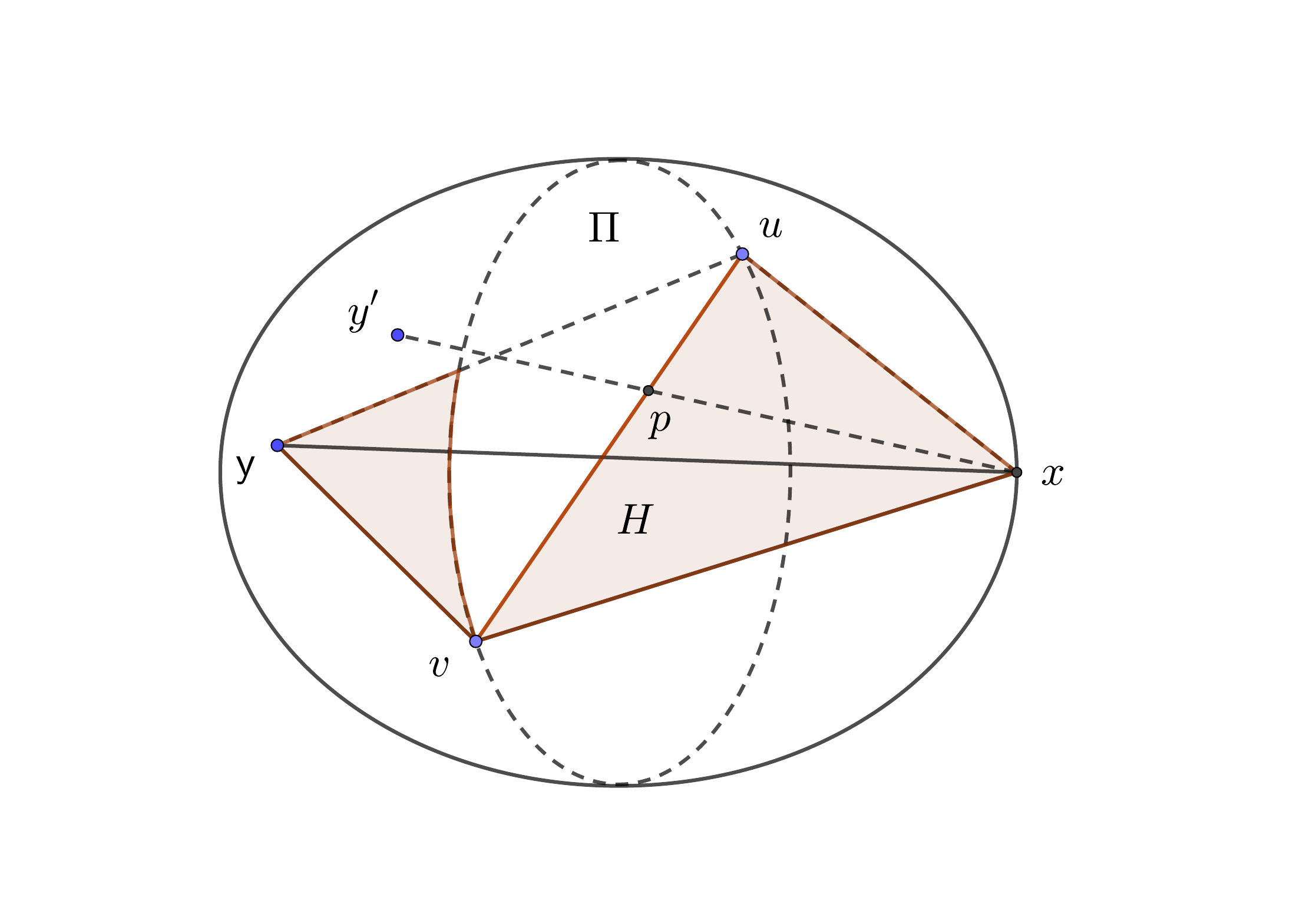}  
\centering
\caption{If $x,y\in \bd K$ and the hyperplane $\Pi$ are such that the relation \\$S(x,L)\cap S(y,L)=\Pi \cap K$ holds, then $x,y$ and $p$ are collinear.}
\label{end}
\end{figure} 

Let $H$ be a hyperplane by $p$. We will prove that for $x\in \bd (H\cap K)$ there exists $y'\in \bd (H\cap K)$ and $(n-2)$-plane $\Pi'$ for $ p $ such that
\begin{eqnarray}\label{adios}
S(H\cap L,x)\cap S(H\cap L,y')=\Pi' \cap (H\cap K). 
\end{eqnarray}
By virtue of the hypothesis there exist $y\in \bd K$ and a hyperplane $\Pi$ by $p$ such that 
\begin{eqnarray}\label{dream}
S(L,x)\cap S(L,y)=\Pi \cap K. 
\end{eqnarray}
By the Lemma \ref{marisol}, $y\in L(p,x)$. Then $y\in H$. If we make $y'=y$ and $\Pi'=\Pi \cap H$ we see that from (\ref{dream}) we have (\ref{adios}). With which the proof of the claim is complete.   

Thus every hyper-sections of $K$ and $L$ passing through $p$ satisfies the conditions of the Theorem \ref{jetta} and, on the other hand, if every $(n-1)$-section of $K$ and $L$ through $p$ are ellipsoids, then $K$ and $L$ are ellipsoids (see \cite{bus}).  Thus by an argument of induction the Theorem \ref{jetta} will follow.

\section{Characterizations of the ellipsoid in terms of a surface of poles}
 \textbf{Proof of Theorem \ref{gina}.} 
In order to prove Theorem \ref{gina} for $n=3$ we will apply Theorem 1 of \cite{gjmc2}. Since is our desire that this work will be self-contained we will present such result. We need the following definition. Let $L$ and $K$ be two $O$-symmetric convex bodies in $\Rt$, with $L\subset  \inte  K$. We say that the points $x, y \in  \bd K$ are free with respect to $L$ if the line through $x$ and $y$, $l(x, y)$, does not meet $L$. Suppose that for every point $x\in  \bd K$, the graze $\Sigma(L,x)$ is a planar curve and denote the plane where it is contained by $\Delta_x$. The body $K$ is said to be almost free with respect to $L$ if for each $z\in \bd K$  and $w\in \Pi_z \cap \bd K$, where $\Pi_z$ is the plane through $O$ parallel to $\Delta_z$, the points $z$ and $w$ are free with respect to $L$ (see Figure \ref{free}).
\begin{figure}[H]
    \centering
    \includegraphics[width=.82\textwidth]{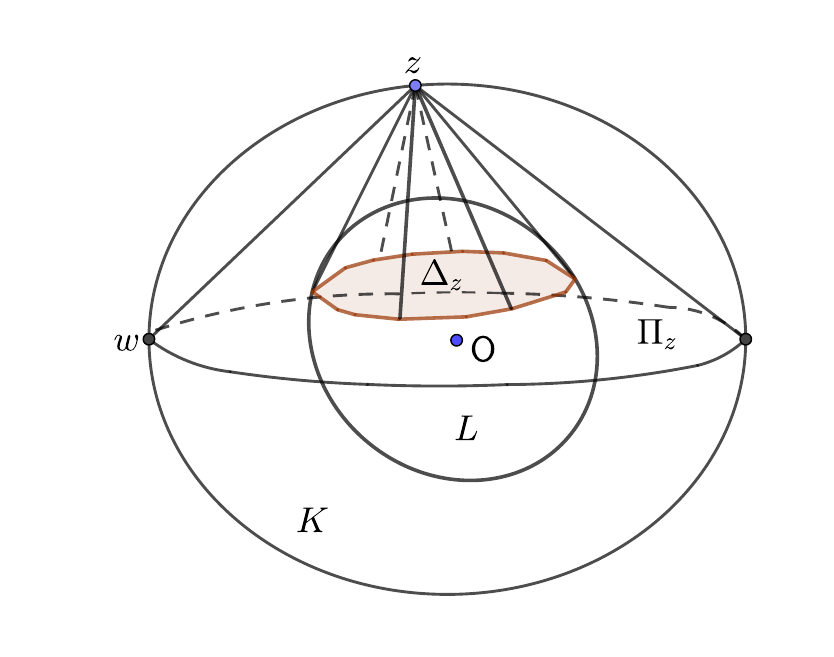}
    \caption{$K$ is almost free with respect to $L$}
    \label{free}
\end{figure}
The Theorem 1 of \cite{gjmc2} of I. Gonzalez-Garc\'ia, J. Jer\'onimo-Castro, E. Morales-Amaya and D. J. Verdusco-Hern\'andez is the following:
 
\textit{Let $K, L \subset \Rt$ be two $O$-symmetric convex bodies with $L\subset \inte K$ strictly convex. Suppose that from every $x$ in $\bd K$ the graze $\Sigma(L, x)$ is a planar curve and $K$ is almost free with respect to $L$. Then $L$ is an ellipsoid}.

For every pole $x\in \bd K$ of $L$ the polar plane will be denoted by $\Delta_x$. On the other hand, since we will only consider grazes of 
 $L$, we just denote by $\Sigma_x$ the graze $\Sigma(L,x)$, for $x\in K$. 
 For $z\in \bd K$ we will denote by $\Pi_z$ the plane through $O$ parallel to $\Delta_z$,
 
\begin{lemma}\label{boquita}
Let $z\in \bd K$. The set $\Omega_z$ is contained in $\Pi_z$.
\end{lemma}
\begin{proof}
Let $\Gamma$ be a plane containing $l(x,-x)$. By the Lemma 1 of \cite{gjmc2}, for $x\in \bd K$, the graze $\Sigma_x$ has center in $l(x,-x)$. The two chords $[a,b]$, $[c,d]$ with boundary points $(\Gamma \cap \Delta_x) \cap \Sigma_x$ and $\Gamma \cap \Omega_x$, respectively, have their mid-points in the line $l(x,-x)$ (by virtue that $\Sigma_x$ has center in $l(x,-x)$ and since $\Omega_x$ is $O$-symmetric). Thus  $[a,b]$ and 
$[c,d]$ are parallel. Hence $[c,d] \subset \Pi_z$. Varying $\Gamma$, such that  $l(x,-x)\subset \Gamma$, it follows that $\Omega_x\subset \Pi_z$.  
 \end{proof}
 
\textbf{Proof of Theorem \ref{gina} for $n=3$}. 
Notice that, since that every point in $x\in \bd K$ is a pole of $L$, the set $\Sigma(L,x)$ is contained in the polar plane, then in order to prove that $L$ is an ellipsoid it will be enough to prove that $L$ is strictly convex and that the body $K$ is almost free with respect to $L$.

The strictly convexity of $L$  follows immediately from the fact that, for every $x\in \bd K$, the set $\Sigma(L,x)$ is contained in the polar plane.

By Lemma \ref{boquita}, by virtue that the condition $\Omega_z\subset \inte K$ holds for $z\in \bd K$, it follows that, for $z\in \bd K$, 
\[
\Omega_z\subset \Pi_z \cap \inte K.
\] 
Thus the points $x,z$ are free with respect to $L$ for every $x\in \Pi_z \cap \bd K$. Hence, by virtue of the arbitrariness of $z\in \bd K$, it follows that the body $K$ is almost free with respect to $L$. Therefore by Theorem 1 of \cite{gjmc2}, $L$ is an ellipsoid.  

\textbf{Proof of Theorem \ref{gina} for $n>3$}. If $n>3$, on the one hand, every hyper-sections of $K$ and $L$ passing through $O$ satisfies the conditions of the Theorem \ref{gina} and, on the other hand, if every $(n-1)$-section  $L$ through $O$ are ellipsoids, then $L$ is an ellipsoid.  Thus by an argument of induction the Theorem \ref{gina} will follow.
\section{Characterization of the elipsoide in terms of ellipsoidal sections of  convex body tangent to a sphere} 
For $u\in \mathbb{S}^{n-1}$, we denote by $H^ {+}(u)$  the closed half-space 
$\{x \in \mathbb{R} ^{n}: x\cdot u \leq 0\}$ with unit normal vector $u$, by $H(u)$ its boundary hyperplane $\{x \in \mathbb{R} ^{n}: \langle x,  u\rangle = 0\}$ and let $S_u:=H(u) \cap \Sd$. For $r\in \mathbb{R}$, we denote by $G(u)$, $rG(u)$  the affine hyperplanes $u+H(u)$, $ru+H(u)$ and by $E(u)$, $rE(u)$, $F(u)$, $rF(u)$ the half-spaces $u+H^+(u)$, $ru+H^+(u)$, $u+H^+(-u)$, $ru+H^+(-u)$, respectively.

A \textit{slab} is a set of the form $\{x \in \mathbb{R} ^{n}: a_1 < x\cdot u  < a_2 \}$, where $u\in \Sd, a_1,a_2\in \R$, $a_1<a_2$. The slab is called an $\epsilon$-slab if $\epsilon=|a_1-a_2|$.

The next result is an auxiliary tool in the proof of Theorem \ref{amanecer}.
\begin{theorem}\label{basico}
Let $K\subset \Rn$ be an strictly convex body, $n\geq 3$, and let $p\in \Rn$. Suppose that, for every hyperplane $\Pi$ through $p$, there exist an 
$\epsilon>0$ and an $\epsilon$-slab $\Delta\subset \Rn$ such that $\Pi \subset \Delta$ and, for each hyperplane $\Gamma\subset \Delta$, the section $\Gamma \cap K$ is centrally symmetric. Then $K$ is an ellipsoid. 
\end{theorem}

 \begin{proof}
By virtue that every orthogonal projection of a centrally symmetric set is centrally symmetric and the well known result \textit{if all the orthogonal projections of a convex body $K\subset \Rn, n\geq 3$, are ellipsoid, then $K$ is an ellipsoid} it is enough to consider the case $n=3$. In fact, if $n>3$, every orthogonal projection of $K$ satisfies the condition of Theorem \ref{basico} in dimension $n-1$. 
 
Case $n=3$. First of all, we observe that every section passing through $p$ is centrally symmetric. If all the sections passing through $p$ have centre at $p$, $K$ has centre $p$. Otherwise, $p$ is a false centre and, by virtue of the False Centre Theorem, $K$ is an ellipsoid and we finish in this case. Hence from now on we will assume that $K$ is centrally symmetric with centre at 
$p$. 

In order to prove that $K$ is an ellipsoid we are going to show that, for every plane $\Pi$ through $p$, the section 
$\Pi \cap K$ is a shadow boundary, which by the Proposition 2 of \cite{Falso_MM} yields that $K$ is an ellipsoid.  

We may assume that $p$ is the origin. Since $K$ is centrally symmetric we can  suppose that each $\epsilon$-slab is symmetric with respect to $p$. Let $\Pi$ be a plane, $p\in \Pi$, and let $\Delta$ be the $\epsilon$-slab corresponding to $\Pi$ given by the hypothesis. Let $\Gamma$ be a plane contained in $\Delta$. 
Notice that a convex set $W\subset \Rn$ is centrally symmetric if and only if $W$ and $-W$ are translated. Thus, since the body $K$ and the sections $\Gamma \cap K$ and $-(\Gamma \cap K)$ are centrally symmetric (but only $K$ has centre at $p$), there exists $u\in \Rt $ such that
\begin{eqnarray}\label{jenny}
-(\Gamma \cap K)=u+(\Gamma \cap K).
\end{eqnarray}
We denote by $D_u$ the cylinder determined by $\Gamma \cap K$ and the lines parallel to $u$. By (\ref{jenny}), it is clear that 
\begin{eqnarray}\label{dyana}
S\partial (K,u)\subset K_u,
\end{eqnarray}
where $K_u:=\bd K \cap (\Rt \backslash D_u)$. 

Let $\{\Gamma_n\}\subset \Delta$ be a sequence of planes such that $\Gamma_n$ is parallel to $\Pi$ for all $n$ and $\Gamma_n \rightarrow \Pi$, when $n\rightarrow \infty$. Let $\{u_n\}\subset \Rn$ such that, for each $u_n$, we have a relation of the type (\ref{jenny}) and (\ref{dyana}), i.e., in particular, the relations
\begin{eqnarray}\label{michel}
S\partial (K,u_n)\subset K_{u_n},
\end{eqnarray}
holds, where $K_{u_n}:=\bd K \cap (\Rt \backslash D_{u_n})$ and $D_{u_n}$ denotes the cylinder determined by $\Gamma_n \cap K$ and the lines parallel to $u_n$. 

We consider the sequence $\{v_n\}\subset \mathbb{S}^2$, where $v_n=\frac{u_n}{||u_n||}$. By the compactness of $\mathbb{S}^2$ there exist $v\subset \mathbb{S}^2$ and a subsequence of $\{v_n\}$, which will be denoted again by $\{v_n\}$, such that $v_n \rightarrow v$, when $n\rightarrow \infty$. It follows that $S\partial (K,v_i) \rightarrow S\partial (K,v) $. On the other hand, by virtue of the strictly convexity of $K$
\begin{eqnarray}\label{barbi}
K_{u_{n+1}}\subset K_{u_n}.
\end{eqnarray}
Furthermore, by the condition $\Gamma_n \rightarrow \Pi$, when $n\rightarrow \infty$, it follows 
\begin{eqnarray}\label{princess}
\Pi \cap K=\bigcap_{n=1}^{\infty} K_{u_n}.
\end{eqnarray}
From (\ref{michel}), (\ref{barbi})  and (\ref{princess}) we deduce that $S\partial (K,v_i) \rightarrow \Pi \cap K$. Therefore $\Pi \cap K=S\partial (K,u)$. 
\end{proof}
 \textbf{Proof of Theorem \ref{amanecer}}
Case $n=3$. We may assume that $p$ is the origin and we suppose that the radius of $B$ is equal to 1, i.e, $\bd B=\Sd$. In order to prove that $K$ is an ellipsoid we are going to use the Theorem \ref{basico}, that is, we are going to show that, for every $u\in \Sd$, there exists $\epsilon>0$ and $\epsilon$-slab $\Delta(u) \subset \Rt$ such that $O\in \Delta(u)$ and, for every plane $rG(u) \subset \Delta(u)$, the section $rG(u) \cap K$ is centrally symmetric. Notice that, in virtue of the hypothesis, there exists a continuous function $\phi:\Sd \rightarrow \Rt$ such that 
\begin{eqnarray}\label{valentina} 
K_u=\phi(u)+K_{-u},
\end{eqnarray}
where $K_u:=G(u)\cap K$.
We are going to show that, for every $u\in \Sd$, the slab $\Delta(u)$ is in fact the slab defined by the planes $G(u)$ and $G(-u)$.  

\textbf{Claim.} Let $u\in \Sd$. For every $v\in S_{\phi(u)}$, the vector $\phi(v)$ is parallel to $H(u)$.

The proof of the Claim follows immediately from the strictly convexity of $K$ and the condition (2) of the Theorem.

Let $v\in S_{\phi(u)}$. We are going to prove that the locus $l_v$ of the mid-points of the chords of the elipse $K_v$ parallel to $G(u)\cap G(v)$ is contained in a line parallel to $\phi(u)$. In order to prove this it is enough to show that the line defined by the mid-point $a,b$ of the chords $G(v) \cap K_u$, $G(v) \cap K_{-u}$ is passing through $C_v$, where $C_v$ is the centre of $K_v$, but such condition follows immediately from (\ref{valentina}), in fact, since  $v\in S_{\phi(u)}$, the relation $G(v) \cap K_{-u}=\phi(u)+(G(v) \cap K_u)$ holds, i.e., $b=\phi(u)+a$.
 
On the other hand, let $r$ be a real number with $0<r<1$. Let $\pi_u:\Rt \rightarrow rG(r)$ be the projection parallel to $u$ onto $rG(u)$. We are going to prove that $rG(u)\cap K$ is centrally symmetric with centre at $\bar{O}:=\pi_u(O)$. Pick a point $x\in \bd(rG(u)\cap K)$. Let $l$ be a supporting line of the elipse $\pi_u(B)$ and let $v\in \Sd$ such that $G(v)=\pi_u^{-1}(l)$. 
  Let $\sigma \in l_{v}, \sigma'\in l_{-v}$ be the mid-points of the chords 
$[x,y]:=rG(u)\cap K_v$ and $[x',y']:=rG(u)\cap K_{-v}$. By (\ref{valentina}) applied to $v$, it follows that $l_{-v}=\phi(v)+l_v$ and by the Claim $x'=\phi(v)+x$, $y'=\phi(v)+y$. Thus $\sigma(-v)=\phi(v)+\sigma(v)$. 
On the other hand, by the previous paragraph, $l_v$ is parallel to $\phi(u)$, furthermore, notice that $l_{-v}=-l_{v}$. Consequently, $\bar {O}$ is the centre of the parallelogram $x,x',y,y'$ and, therefore $||x-\bar{O}||=||y'-\bar{O}||$. 

The case $n>3$ follows immediately from the case $n=3$. Notice that every orthogonal projection of a centrally symmetric set (an ellipsoid) is centrally symmetric (is an ellipsoid) and the well known converse result \textit{if all the orthogonal projections of a convex body $K\subset \Rn, n\geq 3$, are centrally symmetric (ellipsoid), then $K$ is centrally symmetric (an ellipsoid)}. 

Thus, if $n>3$, every orthogonal projection of $K$ satisfies the condition of Theorem \ref{amanecer} in dimension $n-1$ with respect to the orthogonal projection of $\mathbb{S}^{n-1}$.

\end{document}